\documentclass[12pt,reqno]{amsart}

\usepackage{multirow}
\usepackage{hhline}

\usepackage{mathtools}
\usepackage{times}
\usepackage[T1]{fontenc}
\usepackage{mathrsfs}
\usepackage{latexsym}
\usepackage[dvips]{graphics}
\usepackage[titletoc, title]{appendix}
\usepackage{epsfig}
\usepackage{amssymb}
\usepackage{amsmath,amsfonts,amsthm,amssymb,amscd}
\usepackage{color}
\usepackage{hyperref}
\usepackage{amsmath}

\usepackage{color}
\usepackage{breakurl}

\usepackage{comment}
\newcommand{\bburl}[1]{\textcolor{blue}{\url{#1}}}

\newcommand{\be}{\begin{equation}}
\newcommand{\ee}{\end{equation}}

\DeclareMathOperator{\sgn}{sgn}
\newtheorem{thm}{Theorem}[section]

\newtheorem{cor}[thm]{Corollary}

\newtheorem{lem}[thm]{Lemma}
\newtheorem{prop}[thm]{Proposition}
\newtheorem{exa}[thm]{Example}

\newtheorem{defi}[thm]{Definition}
\newtheorem{rek}[thm]{Remark}

\DeclareMathOperator{\supp}{supp}
\DeclareMathOperator{\spann}{span}

\numberwithin{equation}{section}

\title[TGA and Larger Greedy Sums]{Performance of the Thresholding Greedy Algorithm with Larger Greedy Sums}

\author{H\`ung Vi\d{\^e}t Chu}

\email{\textcolor{blue}{\href{mailto:hungchu2@illinois.edu}{hungchu2@illinois.edu}}}
\address{Department of Mathematics, University of Illinois at Urbana-Champaign, Urbana, IL 61820, USA}

\begin{document}

\begin{abstract} 
The goal of this paper is to study the performance of the Thresholding Greedy Algorithm (TGA) when we increase the size of greedy sums by a constant factor $\lambda\geqslant 1$. We introduce the so-called $\lambda$-almost greedy and $\lambda$-partially greedy bases. The case when $\lambda = 1$ gives us the classical definitions of almost greedy and (strong) partially greedy bases. We show that a basis is almost greedy if and only if it is $\lambda$-almost greedy for all (some) $\lambda \geqslant 1$. However, for each $\lambda > 1$, there exists an unconditional basis that is $\lambda$-partially greedy but is not $1$-partially greedy. Furthermore, we investigate and give examples when a basis is 
\begin{enumerate}
    \item  not almost greedy with constant $1$ but is $\lambda$-almost greedy with constant $1$ for some $\lambda > 1$, and 
    \item  not strong partially greedy with constant $1$ but is $\lambda$-partially greedy with constant $1$ for some $\lambda > 1$. 
\end{enumerate}
Finally, we prove various characterizations of different greedy-type bases. 
\end{abstract}


\thanks{The author is thankful to Timur Oikhberg for helpful feedback on earlier drafts of this paper. The author would like to thank the anonymous referee for a careful reading and suggestions that help improve the exposition of this paper.}

\maketitle

\tableofcontents

\section{Introduction}

\subsection{Setup and terminology}
Let $(X,\|\cdot\|)$ be an infinite-dimensional Banach space over $\mathbb{K} = \mathbb{R}$ or $\mathbb{C}$ with the dual $(X^*, \|\cdot\|_{*})$. Let $\mathcal{B} = (e_n)_{n=1}^\infty$ be a countable collection of vectors in $X$ such that 
\begin{enumerate}
    \item There exists a unique collection of biorthogonal functionals $(e^*_n)_{n=1}^\infty$ satisfying $e^*_i(e_j) = \delta_{i,j}$ for all $i, j\in\mathbb{N}$. 
    \item M-boundedness: $\sup_{n} \|e_n\|\|e_n^*\|_{*} < \infty$ and semi-normalization: $$0 \ <\ \inf_{n} \|e_n\|\ \leqslant\ \sup_{n}\|e_n\|\ <\ \infty.$$ 
    \item $X = \overline{\spann\{e_n: n\in \mathbb{N}\}}$.
    \item $X^* = \overline{\spann\{e^*_n: n\in \mathbb{N}\}}^{w^*}$.
\end{enumerate}
Then $\mathcal{B}$ is called a basis.
Note that M-boundedness and semi-normalization imply that $0 < \inf_{n}\|e_n^*\|_{*} \leqslant \sup_n \|e_n^*\|_{*} < \infty$.
Indeed,
for all $m\in\mathbb{N}$, we have $\|e^*_m\|_{*}\|e_m\|\geqslant |e_m^*(e_m)| = 1$, so $\|e_m^*\|_{*}\geqslant \frac{1}{\sup \|e_n\|}$. Moreover, let $c$ be such that $\sup_{n}\|e_n^*\|_{*}\|e_n\|\leqslant c$, then for all $m\in \mathbb{N}$, $\|e_m^*\|_{*}\leqslant \frac{c}{\inf \|e_n\|}$.

With respect to a basis $\mathcal{B}$, every $x\in X$ is uniquely associated with the formal series $\sum_{n=1}^\infty e_n^*(x)e_n$ with $\lim_{n\rightarrow\infty}e_n^*(x)=0$. If the series converges to $x$ for all $x\in X$, we have a Schauder basis. Throughout this paper, we use the following notation
\begin{enumerate}
    \item[i)] For $x\in X$, $\supp(x):= \{n: e_n^*(x)\neq 0\}$ and $\|x\|_\infty := \max_n |e_n^*(x)|$.
    \item[ii)] $X_c := \{x\in X: |\supp(x)|<\infty\}$.
    \item[iii)] $\mathbb N_0 := \mathbb{N}\cup \{0\}$.
    \item[iv)] $\mathbb{N}^{<\infty}:= \{F\subset\mathbb{N}: |F|<\infty\}$.
    \item[v)] For each $A\in \mathbb{N}^{<\infty}$, $$P_A(x) \ :=\ \sum_{n\in A}e_n^*(x)e_n\mbox{ and }P_{A^c}(x) \ :=\ x-P_A(x).$$
    \item[vi)] A sign $\varepsilon = (\varepsilon_n)_{n=1}^\infty$ is a sequence of scalars with modulus $1$.
    \item[vii)] For $z\in \mathbb{K}$, define
    $$\sgn(z) \ :=\ \begin{cases}z/|z|&\mbox{ if }z\neq 0,\\ 1 &\mbox{ if } z = 0.\end{cases}$$
    \item[viii)] For $A\in \mathbb{N}^{<\infty}$ and a sign $\varepsilon$, 
    $$1_A \ :=\ \sum_{n\in A}e_n \mbox{ and } 1_{\varepsilon A}\ :=\ \sum_{n\in A}\varepsilon_n e_n.$$
    \item[ix)] For $A, B\subset \mathbb{N}$ and $x\in X$, we write $A \sqcup B\sqcup x$ to mean that $A, B$, and $\supp(x)$ are pairwise disjoint. We write $A < B$ if $a < b$ for all $a\in A$, $b\in B$. Similar meanings apply for other inequalities.  
\end{enumerate}

We wish to approximate each vector $x\in X$ with a finite linear combination from the basis $\mathcal{B}$. Consider $x \sim \sum_{n=1}^\infty e_n^*(x)e_n$. In 1999, Konyagin and Temlyakov introduced the Thresholding Greedy Algorithm (TGA) that chooses the largest coefficients $e_n^*(x)$ (in modulus)  to be used in the approximation. In particular, a  set $\Lambda\in\mathbb{N}^{<\infty}$ is called a \textit{greedy set} of $x$ if $\min_{n\in\Lambda} |e_n^*(x)|\geqslant \max_{n\notin \Lambda}|e_n^*(x)|$. For $m\in \mathbb{N}_0$, let 
$$\mathcal{G}(x, m)\ :=\ \{\Lambda\subset\mathbb{N}\,:\, \Lambda\mbox{ is a greedy set of }x\mbox{ and }|\Lambda| = m\}.$$
A \textit{greedy operator} of order $m$, denoted by $G_m$, satisfies the following: for $x\in X$,  
$$G_m(x)\ : =\ \sum_{n\in \Lambda_x}e_n^*(x)e_n, \mbox{ for some }\Lambda_x\in\mathcal{G}(x,m).$$
The TGA maps $x$ to a sequence of greedy sums $(G_m(x))_{m=1}^\infty$ for some greedy operators $G_m$.
We capture the error term from the greedy algorithm by 
$$\gamma_m(x)\ :=\ \sup\left\{\|x-P_{\Lambda}(x)\|\,:\, \Lambda\in \mathcal{G}(x, m)\right\},$$
and measure the efficiency of the greedy algorithm by comparing $\gamma_m(x)$ against
the smallest possible error of an arbitrary $m$-term approximation
$$\sigma_m(x) \ :=\ \inf\left\{\left\|x-\sum_{n\in A}a_ne_n\right\|\,:\, A\subset \mathbb{N}, |A|\leqslant m, (a_n)\subset \mathbb{K}\right\}.$$
We also compare $\gamma_m(x)$ against the smallest projection error
$$\widetilde{\sigma}_m(x) \ :=\ \inf\left\{\left\|x-P_A(x)\right\|\,:\, A\subset \mathbb{N}, |A| = m\right\}.$$
More restrictively, we compare $\gamma_m(x)$ against the smallest partial sum error
$$\widehat{\sigma}_m(x)\ :=\ \inf\left\{\|x-S_n(x)\|\,:\, 0\leqslant n\leqslant m\right\},$$
where $S_n(x) := \sum_{i=1}^n e_n^*(x)e_n$.

For conciseness, let ``w.const" stand for ``with constant". Below we define various properties of bases that are of interest. 
\begin{defi}\normalfont
A basis $\mathcal{B}$ is said to be 
\begin{enumerate}
\item \textit{unconditional} if there exists $C\geqslant 1$ such that
$$\|P_A (x)\|\ \leqslant\ C\|x\|,\forall x\in X, \forall A\subset\mathbb{N}.$$
We say that $\mathcal B$ is suppresion unconditional w.const $C$. The least $C$ is denoted by $\mathbf K_s$. When $\mathcal{B}$ is unconditional, there also exists a constant $C\geqslant 1$ such that
$$\left\|\sum_{n=1}^N a_n e_n\right\|\ \leqslant\ C\left\|\sum_{n = 1}^N b_n e_n\right\|, \forall N\in \mathbb{N}, \forall a_n, b_n\in \mathbb{K} \mbox{ with }|a_n|\leqslant |b_n|.$$
We say that $\mathcal{B}$ is unconditional w.const $C$. The least $C$ is denoted by $\mathbf K_u$. It is easy to verify that $\mathbf K_s \leqslant \mathbf K_u\leqslant 2\mathbf K_s$.
\item \textit{quasi-greedy} if there exists $C\geqslant 1$ such that
\begin{equation}
    \label{e5}\gamma_m(x)\ \leqslant\ C\|x\|, \forall x\in X, \forall m\in \mathbb{N}_0.
\end{equation}
In this case, we say that $\mathcal B$ is quasi-greedy or particularly, suppression quasi-greedy w.const $C$. The least such $C$ is denoted by $\mathbf C_\ell$.

It is worth noting that quasi-greedy bases can be defined through the convergence of $(G_m(x))_{m=1}^\infty$ to $x$. By \cite[Theorem 4.1]{AABW}, the two definitions are equivalent.
\item \textit{greedy} if there exists $C\geqslant 1$ such that
\begin{equation}\label{e2}\gamma_m(x)\ \leqslant\ C\sigma_m(x),\forall x\in X, \forall m\in \mathbb{N}_0.\end{equation}
In this case, we say that $\mathcal B$ is greedy w.const $C$. The least $C$ is denoted by $\mathbf C_g$.
\item \textit{almost greedy} if there exists $C\geqslant 1$ such that 
\begin{equation}\label{e3}\gamma_m(x) \ \leqslant\ C\widetilde{\sigma}_m(x), \forall x\in X, \forall m\in \mathbb{N}_0.\end{equation}
We say that $\mathcal B$ is almost greedy w.const $C$. The least such $C$ is denoted by $\mathbf C_a$.
\item \textit{strong partially greedy} if there exists $C\geqslant 1$ such that
\begin{equation}\label{e30}\gamma_m(x) \ \leqslant\ C\widehat{\sigma}_m(x), \forall x\in X, \forall m\in \mathbb{N}_0.\end{equation}
The least such $C$ is denoted by $\mathbf C_p$. Originally, partially greedy bases were introduced by Dilworth et al. \cite{DKKT} for Schauder bases; later, Berasategui et al. defined strong partially greedy bases for general bases \cite{BBL}. While the two concepts coincide for Schauder bases, we shall use the latter as we work with general bases. 
\item \textit{democratic} if there exists $C\geqslant 1$ such that
\begin{equation}\label{e11} \|1_A\|\ \leqslant\ C\|1_B\|,\end{equation}
for all $A, B\in \mathbb{N}^{<\infty}$ with $|A| = |B|$. The smallest such $C$ is denoted by $\mathbf C_{d}$.
\end{enumerate}
\end{defi}

\subsection{Main results}
Greedy bases are desirable because for these bases, an $m$-term approximation by the greedy algorithm produces essentially the best $m$-term approximation. However, it turns out that the greedy condition \eqref{e2} is so strong that it excludes many classical Banach spaces such as $\ell_p\oplus \ell_q (1\leqslant  p < q \leqslant \infty)$,  
$\left(\oplus_{n=1}^\infty \ell_p^n\right)_{\ell_1} (1 < p\leqslant \infty)$, $\left(\oplus_{n=1}^\infty\ell_p^n\right)_{c_0} (1\leqslant p<\infty)$, and $\left(\oplus \ell_p\right)_{\ell_q} (1\leqslant p\neq q < \infty)$ (see \cite[Theorem 1]{DFOS} and \cite[Theorem 1]{Sc}.) However, almost greedy bases are much more common as can be seen from \cite[Remark 7.7]{DKK}, which states that if a Banach space $X$ contains a complemented copy of $\ell_p$ for some $1\leqslant p < \infty$, then $X$ has infinitely many inequivalent almost greedy bases. Surprisingly, almost greedy bases are not far away from being greedy due to the following theorem

\begin{thm}\cite[Theorem 3.3]{DKKT}\label{mt}
Let $\mathcal{B}$ be a basis of a Banach space. The following are equivalent
\begin{enumerate}
    \item[i)] $\mathcal{B}$ is almost greedy. 
    \item[ii)] For some (every) $\lambda > 1$,  there exists a constant $C_{\lambda} \geqslant 1$ such that 
    \begin{equation}\label{e1}\gamma_{\lceil\lambda m\rceil}(x)\ \leqslant\ C_{\lambda}\sigma_m(x), \forall x\in X, \forall m\in \mathbb{N}_0.\end{equation}
\end{enumerate}
\end{thm}
Hence, for an almost greedy basis, a $\lceil \lambda m\rceil$-term approximation by the greedy algorithm gives essentially the best $m$-term approximation. In this sense, one may argue that being almost greedy is not far away from being greedy.

In the same spirit, we would like to know what happens if we enlarge greedy sums in the almost greedy condition \eqref{e3} instead of \eqref{e2}. Does our basis lose the almost greedy property? This motivates the following definition and our first main result.
\begin{defi}\normalfont
Let $\lambda \geqslant 1$. 
A basis $\mathcal B$ is said to be $\lambda$-almost greedy if there exists a constant $C_{\lambda} \geqslant 1$ such that 
     \begin{equation}\label{e4}\gamma_{\lceil\lambda m\rceil}(x)\ \leqslant\ C_{\lambda}\widetilde{\sigma}_m(x), \forall x\in X, \forall m\in \mathbb{N}_0.\end{equation}
The least constant $C_{\lambda}$ is denoted by $\mathbf C_{\lambda, a}$.
\end{defi}

 \begin{thm}\label{em1}
 Let $\mathcal{B}$ be a basis of a Banach space. Then $\mathcal{B}$ is almost greedy if and only if it is $\lambda$-almost greedy for all (some) $\lambda \geqslant 1$.
 \end{thm}
 Therefore, enlarging greedy sums in the almost greedy condition still gives us almost greedy bases. Our second result examines the case for the strong partially greedy condition \eqref{e30}.

\begin{defi}\normalfont
Let $\lambda \geqslant 1$. A basis $\mathcal B$ is said to be $\lambda$-partially greedy if there 
exists $C_\lambda\geqslant 1$ such that 
\begin{equation}\label{e31}\gamma_{\lceil\lambda m\rceil}(x)\ \leqslant\ C_{\lambda}\widehat{\sigma}_m(x), \forall x\in X, \forall m\in \mathbb{N}_0.\end{equation}
The least constant $C_{\lambda}$ is denoted by $\mathbf C_{\lambda, p}$.
\end{defi}

We characterize $\lambda$-partially greedy bases and prove that \eqref{e31} is strictly weaker than 
\eqref{e30} when $\lambda > 1$. To do this, we need the notion of $\lambda$-max conservative.

\begin{defi}\normalfont Let $\lambda\geqslant 1$. 
A basis $\mathcal{B}$ is $\lambda$-max conservative if there exists $C > 0$ such that
$$\|1_A\|\ \leqslant\ C\|1_B\|, \forall A, B\in \mathbb{N}^{<\infty}, A < B, (\lambda - 1)\max A + |A|\leqslant |B|.$$
The least such $C$ is denoted by $\Delta_{\lambda}$.
\end{defi}

\begin{thm}\label{m6}
For $\lambda\geqslant 1$, a basis $\mathcal{B}$ is $\lambda$-partially greedy if and only if it is quasi-greedy and $\lambda$-max conservative.
\end{thm}

\begin{thm}\label{m7} 
Let $\lambda > 1$. The following hold
\begin{enumerate}
\item[i)] If $\mathcal{B}$ is strong partially greedy, then $\mathcal{B}$ is $\lambda$-partially greedy.
\item[ii)] There exists an unconditional basis $\mathcal{B}$ that is $\lambda$-partially greedy but is not strong partially greedy. 
\end{enumerate}
\end{thm}

The problem of studying when different greedy-type constants are equal to $1$ has received much attention recently (see \cite{AA2, AA, AABW, AW, BBL, DK}.)  Our next result give examples of when almost greedy and strong partially greedy bases fail to have the corresponding greedy-type constant $1$, but enlarging greedy sums by a factor $\lambda$ can achieve that. 

\begin{thm}\label{em2}
There exists an unconditional basis that is not almost greedy w.const $1$ but is $\lambda$-almost greedy w.const $1$ for some $\lambda > 1$.
\end{thm}

\begin{thm}\label{em3}
There exists an unconditional basis that is not strong partially greedy w.const $1$ but is $\lambda$-partially greedy w.const $1$ for some $\lambda > 1$. 
\end{thm}

Besides the above main results, we also investigate the problem of enlarging greedy sums in some equivalent formulation of the strong partially greedy property and of the so-called reverse partially greedy property, a probably less prominent greedy-type property.

\subsection{Review of some classical results}
For later reference, we give a quick overview of existing results that characterize different greedy-type bases.
The following is a classical result that characterizes greedy bases.
\begin{thm}[Konyagin and Temlyakov \cite{KT1}]\label{KT1} A basis is greedy if and only if it is unconditional and democratic.
\end{thm}

As observed by Dilworth et al. in \cite[(1)]{DKOSZ}, if $\mathcal{B}$ is unconditional w.const $\mathbf K_u$, suppression unconditional w.const $\mathbf K_s$, and democratic w.const $\mathbf C_d$, then $\mathcal{B}$ is greedy w.const $\mathbf K_s + \mathbf C_d\mathbf K_u^2$. The seemingly current best estimate is $\mathbf K_s + \mathbf C_d\mathbf K_u\min\{\kappa, \mathbf K_u\}$ (see \cite[Theorem 7.2]{AABW}), where $\kappa = \begin{cases}2&\mbox{ if }\mathbb{K} = \mathbb{R},\\ 4 &\mbox{ if }\mathbb{K} = \mathbb{C}\end{cases}$. Hence, even when $\mathbf K_u = \mathbf C_d = 1$, it is not necessary that we have a greedy basis w.const $1$. Indeed, \cite[Theorem E]{DOSZ} gives an unconditional and democratic basis with $\mathbf K_u = \mathbf C_d = 1$, but the greedy constant $\mathbf C_g$ is $2$. To characterize $1$-greedy bases, Albiac and Wojtaszczyk \cite{AW} introduced the so-called symmetry for largest coefficients, which was later generalized by Dilworth et al. \cite{DKOSZ}. 

\begin{defi}\label{AWdefi}\normalfont
A basis $\mathcal B$ is said to be symmetric for largest coefficients (SLC) w.const $C$ if 
\begin{equation}\label{e10}\|x+1_{\varepsilon A}\|\ \leqslant\ C\|x+1_{\delta B}\|,\end{equation}
for all $x\in X$ with $\|x\|_\infty\leqslant 1$, for all $A, B\in\mathbb{N}^{<\infty}$ with $|A|= |B|$ and $A\sqcup B\sqcup x$, and for all signs $\varepsilon, \delta$. The least constant $C$ is denoted by $\mathbf C_b$.
\end{defi}

\begin{thm}[Albiac and Wojtaszczyk \cite{AW}]\label{34AW}
A basis $\mathcal B$ is greedy w.const $1$ if and only if $\mathcal{B}$ is suppression unconditional w.const $1$ and SLC w.const $1$. 
\end{thm}

\begin{thm}\label{Cgreedy}\cite[Theorem 2]{DKOSZ}
Let $\mathcal{B}$ be a basis in a Banach space $X$. 
\begin{enumerate}
    \item[i)] If $\mathcal B$ is greedy w.const $\mathbf C_g$, then it is suppression unconditional w.const $\mathbf C_g$ and is SLC w.const $\mathbf C_g$.
    \item[ii)] Conversely, if $\mathcal B$ is suppression unconditional w.const $\mathbf K_s$ and SLC w.const $\mathbf C_b$, then it is greedy w.const $\mathbf{K}_s\mathbf C_b$. 
\end{enumerate}
\end{thm}

Later, Albiac and Ansorena \cite{AA} characterized almost greedy bases w.const $1$.

\begin{thm}\label{char1ag}\cite[Theorem 2.3]{AA}
A basis in a Banach space is almost greedy w.const $1$ if and only if it is SLC w.const $1$. 
\end{thm}

\begin{thm}\label{AAag}\cite[Theorem 3.3]{AA}
The following hold
\begin{enumerate}
    \item[i)] If $\mathcal{B}$ is almost greedy w.const $\mathbf C_a$, then it is suppression quasi-greedy w.const $\mathbf C_a$ and is SLC w.const $\mathbf C_a$.
    \item[ii)] Conversely, if $\mathcal{B}$ is suppression quasi-greedy w.const $\mathbf{C}_\ell$ and is SLC w.const $\mathbf C_b$, then it is almost greedy w.const $\mathbf{C}_\ell \mathbf C_b$. 
\end{enumerate}
\end{thm}

Recently, Berasategui, Bern\'{a}, and Lassalle \cite{BBL} characterized strong partially greedy bases w.const $1$ by introducing partial symmetry for largest coefficients (PSLC). 
\begin{defi}\normalfont\label{bblpslc}
A basis is said to be partially symmetric for largest coefficients (PSLC) w.const $C$ if 
$$\|x+1_{\varepsilon A}\|\ \leqslant\ C\|x+1_{\delta B}\|,$$
for all $x\in X$ with $\|x\|_\infty\leqslant 1$, for all $A, B\in\mathbb{N}^{<\infty}$ with $|A|\leqslant |B|$ and $A<\supp(x)\sqcup B$, and for all signs $\varepsilon, \delta$. The least constant $C$ is denoted by $\mathbf C_{pl}$.
\end{defi}

\begin{thm}\label{1pg}\cite[Theorem 1.15]{BBL}
A basis $\mathcal{B}$ is strong partially greedy w.const $1$ if and only if $\mathcal{B}$ is PSLC w.const $1$.
\end{thm}

\section{Preliminaries}
In this section, we present several useful results that will be used in due course. 

\subsection{Properties of unconditionality and quasi-greediness}

\begin{prop}\label{pKsKu} Fix $N\in \mathbb{N}$. Let $(e_n)_{n=1}^\infty$ be an unconditional basis with suppression-unconditional constant $\mathbf{K}_s$. For any scalars $a_1, \ldots, a_N, b_1, \ldots, b_N$ so that $|a_n|\leqslant |b_n|$ and $\sgn(a_n) = \sgn(b_n)$ whenever $a_nb_n\neq 0$ for all $1\leqslant n\leqslant N$, we have
$$\left\|\sum_{n=1}^N a_ne_n\right\|\ \leqslant \ \mathbf{K}_s\left\|\sum_{n=1}^N b_n e_n\right\|.$$
\end{prop}

\begin{proof}
\end{proof}

We shall need the uniform boundedness of the truncation operator. For each $\alpha > 0$, we define the truncation function $T_\alpha$ as follows: for  $b\in\mathbb{C}$,
$$T_{\alpha}(b)\ :=\ \begin{cases}\sgn(b)\alpha, &\mbox{ if }|b| > \alpha,\\ b, &\mbox{ if }|b|\leqslant \alpha.\end{cases}$$
We define the truncation operator $T_\alpha: X\rightarrow X$ as
$$T_{\alpha}(x)\ :=\ \sum_{n=1}^\infty T_\alpha(e_n^*(x))e_n \ =\ \alpha 1_{\varepsilon \Gamma_{\alpha}(x)}+ P_{\Gamma_\alpha^c(x)}(x),$$
where $\Gamma_\alpha(x) = \{n: |e_n^*(x)| > \alpha\}$ and $\varepsilon = (\sgn(e_n^*(x)))_{n=1}^\infty$.

\begin{thm}\label{bto}\cite[Lemma 2.5]{BBG} Let $\mathcal{B}$ be suppression quasi-greedy w.const $\mathbf C_\ell$. Then for any $\alpha > 0$, $\|T_\alpha\|\leqslant \mathbf{C}_\ell$.
\end{thm}

\begin{lem}\label{l4}
Let $\lambda \geqslant 1$, $C\geqslant 1$, and $\mathcal{B}$ be a basis that satisfies 
$$\gamma_{\lceil\lambda m\rceil}(x)\ \leqslant\ C\|x\|, \forall x\in X, \forall m\in\mathbb{N}_0.$$
Then $\mathcal{B}$ is quasi-greedy. 
\end{lem}

\begin{proof}
Note that $\{\lceil\lambda m\rceil: m\in \mathbb{N}\}$ is a crude basis and use \cite[Proposition 4.1]{Oi}.
\end{proof}

\begin{defi}\normalfont
A basis is $\lambda$-democratic w.const $C$ if 
\begin{equation}\label{e25} \|1_A\|\ \leqslant\ C\|1_B\|,\end{equation}
for all $A, B\in \mathbb{N}^{<\infty}$ with $\lambda |A| \leqslant |B|$. The smallest $C$ is denoted by $\mathbf C_{\lambda, d}$.
\end{defi}

\begin{lem}\label{l2}
Let $\mathcal{B}$ be a quasi-greedy basis. Pick $\lambda_1, \lambda_2\in [1,\infty)$ with $\lambda_1 < \lambda_2$. Then $\mathcal{B}$ is $\lambda_1$-democratic if and only if $\mathcal{B}$ is $\lambda_2$-democratic.
\end{lem}

\begin{proof}
It suffices to prove that for all $\lambda \geqslant 1$, being $\lambda$-democratic is equivalent to being democratic. Fix $\lambda \geqslant 1$. 

We prove that $\lambda$-democracy implies democracy. Assume that $\mathcal{B}$ is $\lambda$-democratic w.const $\mathbf C_{\lambda,d}$. Let $A, B\in\mathbb{N}^{<\infty}$ with $m:= |A| = |B|$. We consider two cases:
    \begin{enumerate}
        \item[i)] If $m\geqslant 2\lambda$, we partition $A$ into $v:=\lceil 2\lambda\rceil$ subsets, called $A_1, \ldots, A_v$. Then the size of each $A_n$ is 
        $$|A_n|\ \leqslant\ m/(2\lambda) + 1\ \leqslant\ m/\lambda\ =\ |B|/\lambda.$$
        By $\lambda$-democracy, we obtain 
        $$\|1_A\|\ \leqslant\ \sum_{n=1}^v \|1_{A_n}\|\ \leqslant\ v\mathbf C_{\lambda, d}\|1_B\|\ =\ \lceil 2\lambda\rceil\mathbf C_{\lambda, d}\|1_B\|.$$
        \item[ii)] If $m < 2\lambda$, we pick $j\in B$ and have
        $$\|1_{A}\| \ <\ 2\lambda\sup_{n}\|e_n\|\mbox{ and }1 \ =\ |e_j^*(1_B)|\ \leqslant\ \sup_{n}\|e_n^*\|_*\|1_B\|.$$
        Therefore, 
        $$\|1_A\| \ <\ 2\lambda \sup_n \|e_n\|\sup_{n}\|e_n^*\|_*\|1_B\|.$$
    \end{enumerate}
    Hence, $\mathcal{B}$ is democratic. 

Next, we prove that democracy implies $\lambda$-democracy. Assume that $\mathcal{B}$ is democratic w.const $\mathbf C_d$ and suppression quasi-greedy w.const $\mathbf C_\ell$. Let $A, B\in\mathbb{N}^{<\infty}$ with $m:= |A|$ and $|B|\geqslant \lambda m$. Let $D\subset B$ with $|D| = |A| = m$. We have
    $$\|1_A\|\ \leqslant\ \mathbf C_d\|1_D\|\ \leqslant\ \mathbf C_d\mathbf C_\ell \|1_B\|,$$
    for $B\backslash D$ is a greedy set of $1_B$. 
\end{proof}

\subsection{$\lambda$-SLC, $\lambda$-PSLC, and equivalent formulations}
We introduce two useful variations of SLC and PSLC along with their equivalent formulations.

\begin{defi}\label{d1}\normalfont
A basis $\mathcal B$ is said to be $\lambda$-SLC if there exists a constant $C\geqslant 1$ such that
\begin{equation}\label{e8}\|x+1_{\varepsilon A}\|\ \leqslant\ C\|x+1_{\delta B}\|,\end{equation}
for all $x\in X$ with $\|x\|_\infty\leqslant 1$, for all $A, B\in\mathbb{N}^{<\infty}$ with $\lambda |A|\leqslant |B|$ and $A\sqcup B\sqcup x$, and for all signs $\varepsilon, \delta$. The least constant $C$ verifying \eqref{e8} is denoted by $\mathbf C_{\lambda, ib}$.
\end{defi}

\begin{rek}\normalfont\label{r1}
By norm convexity, we have SLC w.const $C$ is equivalent to $1$-SLC w.const $C$. Indeed, one implication is obvious; let us prove that SLC w.const $C$ implies $1$-SLC w.const $C$. Pick $x\in X_c$ with $\|x\|_\infty\leqslant 1$ and $A, B\in\mathbb{N}^{<\infty}$ with $|A| < |B|$ and $A\sqcup B\sqcup x$. Pick signs $\varepsilon, \delta$. Assuming SLC w.const $C$, we want to show that $\|x+1_{\varepsilon A}\|\ \leqslant\ C\|x+1_{\delta B}\|$. Choose $D\in\mathbb{N}^{<\infty}$ with $D > A\cup B\cup \supp(x)$ and $|D| = |B| - |A|$. By norm convexity, 
$$\|x+1_{\varepsilon A}\|\ \leqslant\ \sup_{\theta}\|x+1_{\varepsilon A} + 1_{\theta D}\|\ \leqslant\ C\|x+1_{\delta B}\|,$$
where the second inequality is due to SLC w.const $C$. 
\end{rek}

\begin{prop}\label{ep1}Fix $\lambda\geqslant 1$.
A $\lambda$-SLC basis is $(\lambda^2+ \lambda)$-democratic.
\end{prop}

\begin{proof}Assume that our basis is $\lambda$-SLC w.const $\mathbf C_{\lambda, ib}$.
Let $A, B\in \mathbb{N}^{<\infty}$ with $A\neq \emptyset, (\lambda^2 + \lambda)|A|\leqslant |B|$. Choose $D$ disjoint from $A\cup B$ and $|D| = \lceil \lambda |A|\rceil$. By $\lambda$-SLC, $\|1_A\|\leqslant \mathbf C_{\lambda, ib}\|1_D\|$. Note that
$$\lambda |D|\ =\ \lambda \lceil\lambda |A|\rceil\ \leqslant\ \lambda(\lambda|A| + 1)\ \leqslant\ (\lambda^2 + \lambda)|A|\ \leqslant \ |B|.$$
Again by $\lambda$-SLC, $\|1_D\|\ \leqslant\ \mathbf C_{\lambda, ib}\|1_B\|$. Therefore, $\|1_A\|\ \leqslant\ \mathbf C^2_{\lambda, ib}\|1_B\|$ and so, our basis is $(\lambda^2+ \lambda)$-democratic.
\end{proof}

Let 
$\Omega_{\lambda, ib}\in [1,\infty)$ be the smallest constant (if exists) such that 
\begin{equation}\label{e13}\|x\|\ \leqslant\ \Omega_{\lambda, ib}\|x-P_A(x) + 1_{\varepsilon B}\|\end{equation}
holds for all $x\in X$ with $\|x\|_\infty\leqslant 1$, all signs $\varepsilon$, and all $A, B\in\mathbb{N}^{<\infty}$ with $\lambda |A|\leqslant |B|$ and $(A\cup \supp(x))\cap B = \emptyset$.
\begin{lem}\label{l1}
Let $\mathcal{B}$ be a basis. The following hold
\begin{enumerate}
    \item[i)] If $\mathcal{B}$ satisfies \eqref{e13}, then $\mathcal{B}$ is $\lambda$-SLC w.const $\Omega_{\lambda, ib}$.
    \item[ii)] If $\mathcal{B}$ is $\lambda$-SLC w.const $\mathbf C_{\lambda, ib}$, then $\mathcal{B}$ satisfies \eqref{e13} with 
    $\Omega_{\lambda, ib}\ \leqslant \  \mathbf C_{\lambda, ib}$.
\end{enumerate}
\end{lem}

\begin{proof}
Suppose that $\mathcal{B}$ satisfies \eqref{e13}. Choose $x, A, B, \varepsilon, \delta$ as in Definition \ref{d1}. Let $y = x+1_{\varepsilon A}$. By \eqref{e13}, we have
$$\|x+1_{\varepsilon A}\|\ =\ \|y\| \ \leqslant\ \Omega_{\lambda, ib}\|y - P_A(y) + 1_{\delta B}\|\ =\ \Omega_{\lambda, ib}\|x+1_{\delta B}\|.$$

Conversely, suppose that $\mathcal{B}$ is $\lambda$-SLC w.const $\mathbf C_{\lambda, ib}$. Choose $x, A, B, \varepsilon$ as in \eqref{e13}. We have
\begin{align*}\|x\|\ =\ \left\|x-P_A(x) + \sum_{n\in A}e_n^*(x)e_n\right\|&\ \leqslant\ \sup_{\delta}\left\|x-P_A(x) + 1_{\delta A}\right\|\\
&\ \leqslant\ \mathbf C_{\lambda, ib}\left\|x-P_A(x) + 1_{\varepsilon B}\right\|.
\end{align*}
This completes our proof. 
\end{proof}

The following is a variation of PSLC (Definition \ref{bblpslc}) with the introduction of $\lambda$. 

\begin{defi}\label{d10}\normalfont
A basis is $\lambda$-PSLC if there exists a constant $C\geqslant 1$ such that 
$$\|x+1_{\varepsilon A}\|\ \leqslant\ C\|x+1_{\delta B}\|,$$
for all $x\in X$ with $\|x\|_\infty\leqslant 1$, for all $A, B\in\mathbb{N}^{<\infty}$ with $(\lambda-1)\max A + |A|\leqslant |B|$ and $ A<\supp(x)\sqcup B$, and for all signs $\varepsilon, \delta$. The least constant $C$ is denoted by $\mathbf C_{\lambda, pl}$.
\end{defi}

Let $\Psi_{\lambda, pl}\in [1,\infty)$ be the smallest constant (if exists) such that 
\begin{equation}\label{e32}
    \|x\| \ \leqslant\ \Psi_{\lambda, pl}\|x-P_A(x) + 1_{\varepsilon B}\|
\end{equation}
holds for all $x\in X$ with $\|x\|_\infty\leqslant 1$, for all signs $\varepsilon$, and for all $A, B\in\mathbb{N}^{<\infty}$ with $(\lambda-1)\max A + |A|\leqslant |B|$ and $A < \supp(x-P_A(x))\sqcup B$.

\begin{lem}\label{l10}
Let $\mathcal{B}$ be a basis. The following hold
\begin{enumerate}
    \item[i)] If $\mathcal{B}$ satisfies \eqref{e32}, then $\mathcal{B}$ is $\lambda$-PSLC w.const $\Psi_{\lambda, pl}$.
    \item[ii)] If $\mathcal{B}$ is $\lambda$-PSLC w.const $\mathbf C_{\lambda, pl}$, then $\mathcal{B}$ satisfies \eqref{e32} with 
    $\Psi_{\lambda, pl}\ \leqslant \  \mathbf C_{\lambda, pl}$.
\end{enumerate}
\end{lem}

\begin{proof}
Suppose that $\mathcal{B}$ satisfies \eqref{e32}. Choose $x, A, B, \varepsilon, \delta$ as in Definition \ref{d10}. Let $y = x+ 1_{\varepsilon A}$. By \eqref{e32}, we have
$$\|x+1_{\varepsilon A}\|\ =\ \|y\| \ \leqslant\ \Psi_{\lambda, pl}\|y - P_A(y) + 1_{\delta B}\|\ =\ \Psi_{\lambda, pl}\|x+1_{\delta B}\|.$$

Conversely, suppose that $\mathcal{B}$ is $\lambda$-PSLC w.const $\mathbf C_{\lambda, pl}$. Choose $x, A, B, \varepsilon$ as in \eqref{e32}. We have
\begin{align*}
    \|x\|\ =\ \left\|x-P_A(x) + \sum_{n\in A}e_n^*(x)e_n\right\|&\ \leqslant\ \sup_{\delta}\|x-P_A(x)+1_{\delta A}\|\\
    &\ \leqslant\ \mathbf C_{\lambda, pl}\|x-P_A(x) + 1_{\varepsilon B}\|.
\end{align*}
This completes our proof. 
\end{proof}
\section{Characterizations of $\lambda$-(almost) greedy bases}
The goal of this section is to establish several characterizations of $\lambda$-almost greedy bases as in Theorem \ref{m3}, which contains Theorem \ref{em1}. Furthermore, we construct an unconditional basis that is not almost greedy w.const $1$ but is $\lambda$-almost greedy w.const $1$ for some $\lambda > 1$ (Theorem \ref{em2}.) Along the way, we characterize $\lambda$-greedy w.const 1 and $\lambda$-almost greedy w.const 1 bases, assuming $1$-suppression unconditionality and $1$-suppression quasi-greediness, respectively.

\subsection{Characterizations of $\lambda$-almost greedy bases}
\begin{thm}\label{m2}
Let $\mathcal B$ be suppression quasi-greedy w.const $\mathbf C_\ell$. The following hold.
\begin{enumerate}
    \item[i)] If $\mathcal B$ is $1$-almost greedy w.const $\mathbf C_{1, a}$, then $\mathcal B$ is $1$-SLC w.const $\mathbf C_{1, a}$.
    \item[ii)] If $\mathcal B$ is $\lambda$-almost greedy w.const $\mathbf C_{\lambda, a}$, then $\mathcal B$ is $\lambda$-SLC w.const $\mathbf C_\ell\mathbf C_{\lambda, a}$.
    \item[iii)] If $\mathcal B$ is $\lambda$-SLC w.const $\mathbf C_{\lambda, ib}$, then $\mathcal B$ is $\lambda$-almost greedy w.const $\mathbf C_\ell\mathbf C_{\lambda, ib}$.
\end{enumerate}
\end{thm}

\begin{proof}[Proof of Theorem \ref{m2}]

Item i) follows from Theorem \ref{AAag} and Remark \ref{r1}.

ii) Assume that $\mathcal{B}$ is $\lambda$-almost greedy w.const $\mathbf C_{\lambda, a}$. Let $x, A, B, \varepsilon, \delta$ be chosen as in Definition \ref{d1}. Set $m = |A|$ and $y = x + 1_{\varepsilon A} + 1_{\delta B}$. Choose $\Lambda\subset B$ with $|\Lambda| = \lceil\lambda m\rceil$. Then $\Lambda\in \mathcal{G}(y, \lceil \lambda m\rceil)$. We have
$$\|x+1_{\varepsilon A}\| \ =\ \|y-P_B(y)\|\ \leqslant\ \mathbf C_\ell\|y-P_\Lambda(y)\|\ \leqslant\ \mathbf C_\ell \mathbf C_{\lambda, a}\widetilde{\sigma}_m(y)\ \leqslant \ \mathbf C_\ell \mathbf C_{\lambda, a}\|x+1_{\delta B}\|.$$
Hence, $\mathcal{B}$ is $\lambda$-SLC w.const $\mathbf C_\ell\mathbf C_{\lambda, a}$.

iii) Assume that $\mathcal{B}$ is $\lambda$-SLC w.const $\mathbf C_{\lambda, ib}$. Let $x\in X$, $m\in \mathbb{N}$, $A\in \mathcal{G}(x, \lceil\lambda m\rceil)$, and $B\subset\mathbb{N}$ with $|B| = m$. Let $\alpha := \min_{n\in A}|e_n^*(x)|$. Then $\|x-P_A(x)\|_\infty\leqslant \alpha$. By Lemma \ref{l1}, Theorem \ref{bto}, and noticing that $\lambda|B\backslash A|\leqslant |A\backslash B|$, we have
\begin{align*}
    \|x-P_A(x)\|&\ \leqslant\ \mathbf C_{\lambda, ib}\left\|x-P_A(x)-P_{B\backslash A}(x) + \alpha\sum_{n\in A\backslash B}\sgn(e_n^*(x))e_n\right\|\\
    &\ =\ \mathbf C_{\lambda, ib}\left\|T_\alpha\left(P_{(A\cup B)^c}(x) + \sum_{n\in A\backslash B}e_n^*(x)e_n\right)\right\|\\
    &\ \leqslant\ \mathbf C_\ell\mathbf C_{\lambda, ib}\left\|P_{(A\cup B)^c}(x) + \sum_{n\in A\backslash B}e_n^*(x)e_n\right\|\\
    &\ =\ \mathbf C_\ell\mathbf C_{\lambda, ib}\left\|x-P_B(x)\right\|.
\end{align*}
Taking the sup over all $A\in \mathcal{G}(x, \lceil\lambda m\rceil)$ and the inf over all $B\in\mathbb{N}^{<\infty}$ with $|B| = m$, we see that $\mathcal{B}$ is $\lambda$-almost greedy w.const $\mathbf C_\ell\mathbf C_{\lambda, ib}$.
\end{proof}

\begin{cor}\label{c2}
Let $\mathcal B$ be suppression quasi-greedy w.const $1$. Then $\mathcal{B}$ is $\lambda$-almost greedy w.const $1$ if and only if it is $\lambda$-SLC w.const $1$. 
\end{cor}

\begin{defi}\normalfont
A basis $\mathcal B$ is said to be $\lambda$-greedy if $\mathcal B$ satisfies \eqref{e1}. The least constant $C$ is denoted by $\mathbf C_{\lambda, g}$.
\end{defi}

We state characterizations of almost greedy in the next theorem with our introduction of $\lambda$-SLC, $\lambda$-almost greedy, and $\lambda$-greedy. For previous characterizations of almost greedy bases, see \cite[Theorem 3.3]{DKKT} and \cite[Theorem 6.3]{AABW}.

\begin{thm}\label{m3}
Let $\mathcal{B}$ be a basis. The following are equivalent
\begin{enumerate}
    \item[i)] $\mathcal B$ is almost greedy.
    \item[ii)] $\mathcal B$ is quasi-greedy and democratic.
    \item[iii)] $\mathcal B$ is $\lambda$-greedy for all (some) $\lambda > 1$.
    \item[iv)] $\mathcal B$ is quasi-greedy and SLC.
    \item[v)] $\mathcal B$ is quasi-greedy and is $\lambda$-democratic for all (some) $\lambda \geqslant 1$.
    \item[vi)] $\mathcal B$ is quasi-greedy and $\lambda$-SLC for all (some) $\lambda \geqslant 1$.
    \item[vii)] $\mathcal B$ is $\lambda$-almost greedy for all (some) $\lambda \geqslant 1$.
\end{enumerate}
\end{thm}

The following lemma shall be used in the proof of Theorem \ref{m3}. 
\begin{lem}\label{l3}
Let $\mathcal{B}$ be a quasi-greedy basis. Pick $\lambda_1, \lambda_2\in [1,\infty)$ with $\lambda_1 < \lambda_2$. Then $\mathcal{B}$ is $\lambda_1$-SLC if and only if it is $\lambda_2$-SLC. 
\end{lem}
\begin{proof}
We prove that for all $\lambda > 1$, $\mathcal{B}$ is $\lambda$-SLC if and only if it is SLC. Fix $\lambda > 1$. By Remark \ref{r1}, SLC is the same as $1$-SLC, which implies $\lambda$-SLC. We prove the converse. Assume $\mathcal{B}$ is $\lambda$-SLC, which implies $(\lambda^2+\lambda)$-democracy according to Proposition \ref{ep1}. By Lemma \ref{l2}, $\mathcal{B}$ is democratic. As $\mathcal{B}$ is democratic and quasi-greedy, $\mathcal{B}$ is almost greedy by \cite[Theorem 3.3]{DKKT} and thus, is SLC according to Theorem \ref{AAag}.
\end{proof}

\begin{proof}[Proof of Theorem \ref{m3}] 
The equivalence among (i), (ii), and (iii) is a consequence of \cite[Theorem 3.3]{DKKT}. We deduce from Theorem \ref{AAag} that (i) and (iv) are equivalent. In turn, (ii) and (v) are equivalent by Lemma \ref{l2}. Finally, (iv) and (vi) are equivalent by Lemma \ref{l3}. To complete the proof, we shall show that (i) implies (vii), and (vii) implies (iv). 

To prove that (i) implies (vii), we fix $\lambda_0\geqslant 1$ and show that there exists $C\geqslant 1$ such that
\begin{equation}\label{e26}\gamma_{\lceil \lambda_0 m\rceil}(x) \ \leqslant\ C\widetilde{\sigma}_m(x), \forall x\in X, \forall m\in \mathbb{N}_0.\end{equation}
If $\lambda_0 = 1$, then \eqref{e26} is exactly the definition of almost greedy bases and we are done. Suppose that $\lambda_0 > 1$. Since (i) is equivalent to (iii), we know that $\mathcal{B}$ is $\lambda_0$-greedy; that is, there exists some $C'\geqslant 1$ such that 
\begin{equation}\label{e27}\gamma_{\lceil \lambda_0 m\rceil}(x) \ \leqslant\ C'\sigma_m(x), \forall x\in X, \forall m\in \mathbb{N}_0.\end{equation}
As \eqref{e27} implies \eqref{e26} with $C = C'$, we are done. 

It remains to show that (vii) implies (iv). Assume that $\mathcal{B}$ is $\lambda_0$-almost greedy for some $\lambda_0 \geqslant 1$. Then there exists $C\geqslant 1$ such that 
\begin{equation}\label{e28}\gamma_{\lceil \lambda_0 m\rceil}(x) \ \leqslant\ C\widetilde{\sigma}_m(x), \forall x\in X, \forall m\in \mathbb{N}_0.\end{equation}
Since $\widetilde{\sigma}_m(x)\leqslant \|x\|$, we obtain
$$\|x-P_{\Lambda}(x)\|\ \leqslant\ C\|x\|, \forall x\in\mathbb{N}, \forall m\in\mathbb{N}_0, \forall \Lambda\in \mathcal{G}(x, \lceil\lambda_0 m\rceil),$$
which implies that $\mathcal{B}$ is quasi-greedy by Lemma \ref{l4}. On the other hand, by Theorem \ref{m2}, $\mathcal{B}$ is  $\lambda_0$-SLC, which implies SLC according to Lemma \ref{l3} and Remark \ref{r1}.
This completes our proof.
\end{proof}

\subsection{$\lambda$-(almost) greedy bases w.const $1$}

\begin{proof}[Proof of Theorem \ref{em2}]
Let $\omega = (w(n))_{n=1}^\infty\subset \mathbb{R}$ such that $0 < \inf_n w(n) < \sup_n w(n) < \infty$. We renorm $\ell_1$ as follows: for $x = (x_1, x_2, x_3, \ldots)\in \ell_1$, let 
$$\|x\| \ =\ \sum_{n=1}^\infty w(n)|x_n|.$$
Let $\mathcal{B} = (e_n)_{n=1}^\infty$ be the canonical basis. Under $\|\cdot\|$, $\mathcal{B}$ is not SLC w.const $1$. Indeed, since $\inf_n w(n) < \sup_n w(n)$, we choose $i, j\in \mathbb{N}$ such that $w(i) < w(j)$. Then $\|e_i\| = w(i) < w(j) = \|e_j\|$. By Theorem \ref{char1ag}, $\mathcal B$ is not almost greedy w.const $1$. However, if we set $\lambda =  \frac{\sup_n w(n)}{\inf_n w(n)}$, then we can prove that $\mathcal{B}$ satisfies
\begin{equation}\label{e20}\gamma_{\lceil\lambda m\rceil}(x)\ \leqslant\ \sigma_m(x), \forall x\in \ell_1, \forall m\in \mathbb{N}_0.\end{equation}
Indeed, pick $x = (x_1, x_2, x_3, \ldots)\in\ell_1$, $m\in \mathbb{N}$, $A \in \mathcal G(x, \lceil\lambda m\rceil)$, and $B\subset\mathbb{N}$ with $|B| = m$. Also, pick $(b_n)_{n\in B}\subset\mathbb{K}$. Set $t := \min_{n\in A}|x_n|$.
Since $A$ is a greedy set and $|A|\geqslant \lambda |B|$ implies $|A\backslash B|\geqslant \lambda |B\backslash A|$, we have  
\begin{equation}\label{e21}
    \|P_{B\backslash A}(x)\|\ \leqslant\ t\sup_{n}w(n)|B\backslash A|\ \leqslant\ \frac{t}{\lambda}\sup_{n}w(n)|A\backslash B|.
\end{equation}
Furthermore, 
\begin{equation}\label{e22}
    |A\backslash B|\ \leqslant\ \frac{1}{\inf_n w(n)}\sum_{n\in A\backslash B}w(n)\cdot 1\ =\   \frac{1}{\inf_n w(n)} \|1_{A\backslash B}\|.
\end{equation}
We obtain 
\begin{eqnarray*}
    \|x-P_A(x)\|&\ \leqslant&\ \|x-P_A(x)\| - \|P_{B\backslash A}(x)\| + \frac{t}{\lambda}\sup_{n}w(n)||A\backslash B|\\
    &\ \leqslant&\ \|x-P_A(x)\| - \|P_{B\backslash A}(x)\| + \|t1_{A\backslash B}\|\\
    &\ =&\ \|x-P_{A\cup B}(x) + t1_{A\backslash B}\|\\
    &\ \leqslant&\ \left\|x-P_{A\cup B}(x) + P_{A\backslash B}(x) + \sum_{n\in B}(x_n - b_n)e_n\right\|\\
    &\ = &\ \left\|x-\sum_{n\in B}b_n e_n\right\|,
\end{eqnarray*}
where we have used \eqref{e21} to deduce the first inequality; the second one follows from \eqref{e22}; the third inequality is due to the fact that $\mathcal B$ is unconditional w.const $\mathbf K_u = 1$.
Since $A, B$ and $(b_n)_{n\in B}$ are arbitrarily chosen, we complete our proof of \eqref{e20}.
\end{proof}

\begin{rek}\normalfont
The above renorming of $\ell_1$ indicates that though a basis may not be almost greedy w.const $1$, it can still give the best $m$-term approximation if we allow greedy sums to be of size greater than $m$ by a constant factor. 
\end{rek}

We have an analog of Theorem \ref{m2} for the $\lambda$-greedy case.
\begin{thm}\label{m1}
Let $\mathcal B$ be suppression unconditional w.const $\mathbf K_s$. The following hold
\begin{enumerate}
    \item[i)] If $\mathcal B$ is $1$-greedy w.const $\mathbf C_{1, g}$, then $\mathcal B$ is $1$-SLC w.const $\mathbf C_{1, g}$. 
    \item[ii)] If $\mathcal B$ is $\lambda$-greedy w.const $\mathbf C_{\lambda, g}$, then $\mathcal B$ is $\lambda$-SLC w.const $\mathbf K_s\mathbf C_{\lambda, g}$.
    \item[iii)] If $\mathcal B$ is $\lambda$-SLC w.const $\mathbf C_{\lambda, ib}$, then $\mathcal B$ is $\lambda$-greedy w.const $\mathbf K_s\mathbf C_{\lambda, ib}$.
\end{enumerate}
\end{thm}

\begin{proof}
Item i) follows from Theorem \ref{Cgreedy}.

We prove ii). Assume that $\mathcal{B}$ is $\lambda$-greedy w.const $\mathbf C_{\lambda, g}$. Let $x, A, B, \varepsilon, \delta$ be chosen as in Definition \ref{d1}. Set $m = |A|$ and $y = x + 1_{\varepsilon A} + 1_{\delta B}$. Choose $\Lambda\subset B$ with $|\Lambda| = \lceil\lambda m\rceil$. Then $\Lambda\in \mathcal{G}(y, \lceil \lambda m\rceil)$. We have
$$\|x+1_{\varepsilon A}\| \ =\ \|y-P_B(y)\|\ \leqslant\ \mathbf K_s\|y-P_\Lambda(y)\|\ \leqslant\ \mathbf K_s \mathbf C_{\lambda, g}\sigma_m(y)\ \leqslant \ \mathbf K_s \mathbf C_{\lambda, g}\|x+1_{\delta B}\|.$$
Hence, $\mathcal{B}$ is $\lambda$-SLC w.const $\mathbf K_s\mathbf C_{\lambda, g}$.

We prove iii). Assume that $\mathcal{B}$ is $\lambda$-SLC w.const $\mathbf C_{\lambda, ib}$. Let $x\in X$, $m\in \mathbb{N}$, $A\in \mathcal{G}(x, \lceil\lambda m\rceil)$, and $B\subset\mathbb{N}$ with $|B| = m$. Let $(b_n)_{n\in B}\subset\mathbb{K}$ and $\alpha := \min_{A}|e_n^*(x)|$. Then $\|x-P_A(x)\|_\infty\leqslant \alpha$. By Lemma \ref{l1}, Proposition \ref{pKsKu}, and noticing that $\lambda|B\backslash A|\leqslant |A\backslash B|$, we have
\begin{align*}
    \|x-P_A(x)\|&\ \leqslant\ \mathbf C_{\lambda, ib}\left\|x-P_A(x)-P_{B\backslash A}(x) + \alpha\sum_{n\in A\backslash B}\sgn(e_n^*(x))e_n\right\|\\
    &\ =\ \mathbf C_{\lambda, ib}\left\|P_{(A\cup B)^c}(x) + \alpha\sum_{n\in A\backslash B}\sgn(e_n^*(x))e_n\right\|\\
    &\ \leqslant\ \mathbf K_s\mathbf C_{\lambda, ib}\left\|P_{(A\cup B)^c}(x) + \sum_{n\in B}(e_n^*(x)-b_n)e_n + \sum_{n\in A\backslash B}e_n^*(x)e_n\right\|\\
    &\ =\ \mathbf K_s\mathbf C_{\lambda, ib}\left\|x-\sum_{n\in B}b_ne_n\right\|.
\end{align*}
Taking the sup over all $A\in \mathcal{G}(x, \lceil\lambda m\rceil)$ and the inf over all $B\subset\mathbb{N}$ with $|B| = m$ and over all $(b_n)_{n\in B}\subset\mathbb{K}$, we see that $\mathcal{B}$ is $\lambda$-greedy w.const $\mathbf K_s\mathbf C_{\lambda, ib}$.
\end{proof}

\begin{cor}\label{c1}
Let $\mathcal B$ be suppression unconditional w.const $1$. Then $\mathcal{B}$ is $\lambda$-greedy w.const $1$ if and only if it is $\lambda$-SLC w.const $1$. 
\end{cor}

\begin{exa}\normalfont
We consider the renorming of $\ell_1$ used in the proof of Theorem \ref{em2}. Clearly, the canonical basis $\mathcal{B}$ is suppression unconditional w.const $\mathbf K_s = 1$. Let $\lambda = \frac{\sup_{n}w(n)}{\inf_n w(n)}$. We show that $\mathcal{B}$ is $\lambda$-SLC w.const $1$. Choose $x, A, B, \varepsilon, \delta$ as in Definition \ref{d1}. We have
$$\|1_{\varepsilon A}\|\ =\ \sum_{n\in A}w(n) \ \leqslant\ |A|\sup_{n}w(n)\ \leqslant\ \frac{\sup_nw(n)}{\lambda}|B|\ =\ \inf_n w(n)|B|\ \leqslant\ \|1_{\delta B}\|.$$
Therefore, 
$$\|x+1_{\varepsilon A}\|\ = \|x\| + \|1_{\varepsilon A}\|\ \leqslant\ \|x\| + \|1_{\delta B}\|\ =\ \|x+1_{\delta B}\|.$$
By Corollary \ref{c1}, we know that $\mathcal{B}$ is $\lambda$-greedy w.const $1$ (though it is not almost greedy w.const $1$.)
\end{exa}

\section{On the $\lambda$-partially greedy property}
There are three main results to be proved in this section. First, we characterize $\lambda$-partially greedy bases as being quasi-greedy and $\lambda$-max conservative (Theorem \ref{m6}.) Next, we construct an unconditional basis $\mathcal{B}$ that is $\lambda$-partially greedy but is not strong partially greedy (Theorem \ref{m7}.) Finally, we prove Theorem \ref{em3} and characterize $\lambda$-partially greedy w.const $1$ bases, assuming $1$-suppression quasi-greediness (Corollary \ref{c11}.) 

\subsection{Characterization of the $\lambda$-partially greedy property}
\begin{thm}\label{m4}
Let $\mathcal B$ be suppression quasi-greedy w.const $\mathbf C_\ell$. The following hold
\begin{enumerate}
    \item[i)] If $\mathcal B$ is $1$-partially greedy w.const $\mathbf C_{1, p}$, then $\mathcal B$ is $1$-PSLC w.const $\mathbf C_{1, p}$.
    \item[ii)] If $\mathcal B$ is $\lambda$-partially greedy w.const $\mathbf C_{\lambda, p}$, then $\mathcal B$ is $\lambda$-PSLC w.const $\mathbf C_\ell\mathbf C_{\lambda, p}$.
    \item[iii)] If $\mathcal B$ is $\lambda$-PSLC w.const $\mathbf C_{\lambda, pl}$, then $\mathcal B$ is $\lambda$-partially greedy w.const $\mathbf C_\ell\mathbf C_{\lambda, pl}$.
\end{enumerate}
\end{thm}

\begin{proof}
Item i) follows from \cite[Theorem 4.2]{B0}.

ii) Let $x, A, B, \varepsilon, \delta$ be chosen as in Definition \ref{d10}. Let $m:=\max A$ and $D = \{1, \ldots, m\}\backslash A$. It follows that 
$$|B\cup D|\ =\ |B| + |D| \ \geqslant\ (\lambda -1)m + |A| + m - |A|\ =\ \lambda m.$$
Choose $\Lambda\subset B\cup D$ with $|\Lambda| = \lceil\lambda m\rceil$. 
Set 
$$y\ := \ 1_D + 1_{\varepsilon A} + x + 1_{\delta B}.$$
Then $\Lambda\in \mathcal{G}(y, \lceil\lambda m\rceil)$.
We have
\begin{align*}
    \|1_{\varepsilon A} + x\|\ =\ \|y- P_{B\cup D}(y)\|&\ \leqslant\ \mathbf C_\ell\|y- P_{\Lambda}(y)\|\\
    &\ \leqslant\ \mathbf C_\ell\mathbf C_{\lambda, p}\widehat{\sigma}_m(y)\ \leqslant\ \mathbf C_\ell\mathbf C_{\lambda, p}\|x+1_{\delta B}\|.
\end{align*}

iii)  Assume that $\mathcal{B}$ is $\lambda$-PSLC w.const $\mathbf C_{\lambda, pl}$. Let $x\in X$, $m\in \mathbb{N}_0$ and $A\in \mathcal G(x, \lceil\lambda m\rceil)$. Fix $k\in [0,m]$. We need to show that 
$$\|x-P_A(x)\|\ \leqslant\ \mathbf C_\ell \mathbf C_{\lambda, pl}\|x-S_k(x)\|.$$
Set $B = \{1, 2, \ldots, k\}\backslash A$, $F = A\backslash \{1, 2, \ldots, k\}$, and $\alpha:= \min_{n\in A}|e_n^*(x)|$. Clearly, $\|x-P_A(x)\|_\infty\leqslant \alpha$. Observe that 
\begin{align*}
    |B| + (\lambda-1)\max B&\ \leqslant\ k - |\{1, 2, \ldots, k\}\cap A| + (\lambda-1)k\\
    &\ \leqslant \ \lambda m - |\{1, 2, \ldots, k\}\cap A|\\
    &\ \leqslant\ |A| - |\{1, 2, \ldots, k\}\cap A|\ =\ |F|.
\end{align*}
Furthermore, 
$$ B \ < \ \supp(x-P_A(x)-P_B(x))\sqcup F.$$
By Lemma \ref{l10} and Theorem \ref{bto}, we obtain
\begin{align*}
    \|x-P_A(x)\|&\ \leqslant\ \mathbf C_{\lambda, pl}\left\|x-P_A(x)-P_B(x) + \alpha\sum_{n\in F}\sgn(e_n^*(x))e_n\right\|\\
    &\ =\ \mathbf C_{\lambda, pl}\left\|T_\alpha\left(x-P_A(x)-P_B(x) + P_F(x)\right)\right\|\\
    &\ \leqslant\ \mathbf C_\ell\mathbf C_{\lambda, pl}\|x-S_k(x)\|,
\end{align*}
as desired.
\end{proof}

\begin{cor}\label{el1}Fix $\lambda\geqslant 1$.
A basis $\mathcal{B}$ is $\lambda$-partially greedy if and only if it is quasi-greedy and $\lambda$-PSLC. 
\end{cor}

\begin{proof}
Fix $\lambda \geqslant 1$. Assume that $\mathcal{B}$ is $\lambda$-partially greedy. Then there exists $C\geqslant 1$ such that 
$$\gamma_{\lceil\lambda m\rceil}(x)\ \leqslant\ C\widehat{\sigma}_m(x), \forall x\in X, \forall m\in \mathbb{N}_0.$$
Hence, 
$$\gamma_{\lceil\lambda m\rceil}(x)\ \leqslant\ C\|x\|,\forall x\in X, \forall m\in \mathbb{N}_0.$$\
By Lemma \ref{l4}, $\mathcal{B}$ is quasi-greedy. We now use Theorem \ref{m4} to conclude that $\mathcal{B}$ is $\lambda$-PSLC.

Now assume that $\mathcal{B}$ is quasi-greedy and $\lambda$-PSLC. Again by Theorem \ref{m4}, $\mathcal{B}$ is $\lambda$-partially greedy.
\end{proof}

We will use the next lemma in the proof of Theorem \ref{m6}. 

\begin{lem}\label{l40}
If $\mathcal{B}$ is quasi-greedy, then $\mathcal{B}$ is $\lambda$-PSLC if and only if it is $\lambda$-max conservative.
\end{lem}

\begin{proof}
By setting $x = 0$ in Definition \ref{d10}, we see that $\lambda$-PSLC implies $\lambda$-max conservative. Assume that $\mathcal{B}$ is $\lambda$-max conservative w.const $\Delta_\lambda$ and is suppression quasi-greedy w.const $\mathbf C_\ell$. Let $x, A, B, \varepsilon, \delta$ be chosen as in Definition \ref{d10}. We have
\begin{align*}
\|x+1_{\varepsilon A}\|\ \leqslant\ \|x\| + \|1_{\varepsilon A}\|&\ \leqslant\ \mathbf C_\ell \|x+1_{\delta B}\| + \Delta_{\lambda}\|1_{\delta B}\|\\
&\ \leqslant\ \mathbf C_\ell\|x+1_{\delta B}\| + \Delta_{\lambda}(1+\mathbf C_\ell)\|x+1_{\delta B}\|.
\end{align*}
The first inequality is due to the triangle inequality, while the second and third inequalities is due to the fact that $B$ is a greedy set of $x + 1_{\delta B}$ and due to $\lambda$-max conservative. Hence, 
$$\|x+1_{\varepsilon A}\| \ \leqslant\ (\mathbf C_\ell + \Delta_\lambda + \Delta_\lambda\mathbf C_\ell)\|x+1_{\delta B}\|.$$
This completes our proof.
\end{proof}

\begin{proof}[Proof of Theorem \ref{m6}]
If $\mathcal{B}$ is $\lambda$-partially greedy, then $\mathcal{B}$ is quasi-greedy and $\lambda$-PSLC by Corollary \ref{el1}. According to Lemma \ref{l40}, $\mathcal{B}$ is quasi-greedy and $\lambda$-max conservative. 

Conversely, if $\mathcal{B}$ is quasi-greedy and $\lambda$-max conservative, then Lemma \ref{l40} states that $\mathcal{B}$ is $\lambda$-PSLC. Hence, by Corollary \ref{el1}, $\mathcal{B}$ is $\lambda$-partially greedy. 
\end{proof}

\subsection{For $\lambda > 1$, $\lambda$-partially greedy is strictly weaker than strong partially greedy}
The goal of this subsection is to prove Theorem \ref{m7}.

\begin{defi}\label{d50}\normalfont
A basis $\mathcal{B}$ is $\lambda$-conservative w.const $C$ if for all $A, B\in\mathbb{N}^{<\infty}$ with $\lambda |A| \leqslant |B|$ and $A < B$, we have
$$\|1_A\|\ \leqslant\ C\|1_B\|.$$
The least constant $C$ is denoted by $\mathbf C_{\lambda, c}$. When $\lambda = 1$, we say that our basis is conservative and set $\mathbf C_{c} := \mathbf C_{1,c}$.
\end{defi}

\begin{rek}\label{r2}\normalfont
Note that by definition, $1$-max conservative is precisely conservative. Hence, by Theorem \ref{m6}, a basis is strong partially greedy if and only if it is quasi-greedy and conservative. (See also \cite[Theorem 4.2]{B0}.)
\end{rek}

\begin{proof}[Proof of Theorem \ref{m7}] Choose $\lambda > 1$.

i) Assume that $\mathcal{B}$ is strong partially greedy; that is, there exists $C\geqslant 1$ such that 
    $$\gamma_m(x)\ \leqslant\ C\widehat{\sigma}_m(x)\ \leqslant\ C\|x\|, \forall x\in X, \forall m\in \mathbb{N}_0.$$
    Hence, $\mathcal{B}$ is suppression quasi-greedy w.const $C$. Therefore, 
    $$\gamma_{\lceil\lambda m\rceil}(x) \ \leqslant\ C\gamma_m(x)\ \leqslant\ C^2\widehat{\sigma}_m(x),\forall x\in X, \forall m\in \mathbb{N}_0.$$
    We conclude that $\mathcal{B}$ is $\lambda$-partially greedy.

ii) For each $\lambda > 1$, we now construct an unconditional basis $\mathcal{B}$ that is $\lambda$-partially greedy but is not strong partially greedy. 

Step 1: Our construction is a modification of the Schreier space. Let $\mathcal{F} = \{F\subset\mathbb{N}\,:\, |F|\geqslant 1\mbox{ and }\min F\geqslant |F|\}$ and $D = \{2^n\,:\, n\in \mathbb{N}_0\}$. Define weight sequences $\omega^F = (w^F_n)_{n=1}^\infty$ that depend on $F\subset\mathbb{N}$
$$w^F_n\ :=\ \begin{cases}\frac{1}{\sqrt{n}}&\mbox{ if }F\subset D,\\\frac{1}{n}&\mbox{ otherwise}.\end{cases}$$
Let $X$ be the completion of $c_{00}$ with respect to the following norm: for $x = (x_1, x_2, \ldots)$, 
$$\|x\|\ :=\ \sup\left\{\sum_{n\in F}w^F_{\sigma(n)}|x_n|\,:\, F\in \mathcal{F}\mbox{ and } \sigma: F\rightarrow [1, |F|]\mbox{ is a bijection}\right\}.$$
Clearly, the canonical basis $\mathcal{B}$ is a unconditional w.const $\mathbf K_u = 1$ and normalized. 

Step 2: We show that $\mathcal{B}$ is not conservative and thus, not strong partially greedy by Remark \ref{r2}. Fix $k\in \mathbb{N}$. Choose $N$ sufficiently large such that 
$A:=\{2^{N+1}, \ldots, 2^{N+k}\}$ is in $\mathcal{F}$ and $B:= \{2^{N+k}+1, \ldots, 2^{N+k}+k\}\cap D = \emptyset$. Then 
$\|1_A\|\sim \sqrt{k}$ and $\|1_B\|\sim \ln k$. Hence, $\mathcal{B}$ is not conservative. 

Step 3: We prove that $\mathcal{B}$ is $\lambda$-max conservative. Let $A, B\in\mathbb{N}^{<\infty}$ with $A<B$ and $|A| + (\lambda - 1)\max A\leqslant |B|$. We need to  show that $\|1_A\|\ \lesssim\ \|1_B\|$. Let $F\subset A$, $F\neq \emptyset$, and $F\in \mathcal{F}$. 

If $F\not\subset D$, then for all $\sigma: F\rightarrow [1, |F|]$, 
$$\sum_{n\in F}w^F_{\sigma(n)}\ =\ \sum_{n=1}^{|F|}\frac{1}{n}\ \leqslant\ \ln(|F|) + 1.$$
Choose $B'\subset B$ such that $B'\in \mathcal{F}, B'\neq \emptyset$, and $|B'|\geqslant |B|/2$. We have 
$$\|1_B\|\ \geqslant\ \sum_{n=1}^{|B'|}\frac{1}{n}\ \geqslant\ \ln(|B'|)\ \geqslant\ \ln(|B|)-\ln(2).$$
Using $\|1_B\|\geqslant 1$, we obtain
$$\sum_{n\in F}w^F_{\sigma(n)}\ \leqslant\ \ln(|F|) + 1\ \leqslant\ \ln(|B|) + 1 \ \leqslant\ \|1_B\| + 1 + \ln(2)\ \leqslant\ (2+\ln(2))\|1_B\|.$$

Assume that $F\subset D$. Write $$F \ =\  \{2^{n_1}, 2^{n_2}, \ldots, 2^{n_j}\}.$$
Then $\|1_F\|\sim \sqrt{j}$. Choose $B'\subset B$ such that $B'\in \mathcal{F}$ and $|B'|\geqslant |B|/2$. We have
\begin{align}\label{e61}\|1_B\|\ \geqslant\ \|1_{B'}\|\ \geqslant\ \ln |B'|&\ \geqslant\ \ln|B|-\ln 2\nonumber\\
&\ \geqslant\ \ln((\lambda-1)2^{n_j})-\ln(2)\nonumber\\
&\ =\ n_j\ln 2+\ln((\lambda-1)/2)\nonumber\\
&\ \geqslant\ j\ln 2 + \ln((\lambda-1)/2)\nonumber\\
&\ \gtrsim\ \|1_F\|\ln 2 + \ln((\lambda-1)/2).
\end{align}

Case 1: $\frac{1}{2}\|1_F\| + \ln((\lambda-1)/2)\geqslant 0$. Then \eqref{e61} implies that
$$\|1_B\|\ \gtrsim\ (\ln2 - 0.5)\|1_F\|.$$

Case 2: $\frac{1}{2}\|1_F\| + \ln((\lambda-1)/2) < 0$. Then 
$$\|1_F\| \ <\ 2\ln(2/(\lambda-1))\ \leqslant\ 2\ln(2/(\lambda-1))\|1_B\|.$$

Since $F$ is arbitrary, it holds that $\|1_A\|\lesssim \|1_B\|$ and so, $\mathcal{B}$ is $\lambda$-max conservative, as desired. Since $\mathcal{B}$ is unconditional, we conclude that $\mathcal{B}$ is $\lambda$-partially greedy by Theorem \ref{m6}. However, $\mathcal{B}$ is not strong partially greedy by Remark \ref{r2}. 
\end{proof}

\subsection{$\lambda$-partially greedy bases w.const $1$}

\begin{proof}[Proof of Theorem \ref{em3}]
We consider a modification of the Schreier space as in  \cite[Proposition 6.10]{BDKOW} and a renorming using bounded weights. Let $X$ be the completion of $c_{00}$ under the following norm
$$\|(x_1, x_2, x_3, \ldots)\|_{X}\ =\ \sup_{F\in \mathcal{F}}\sum_{n\in F}|x_n|,$$
where $\mathcal{F} = \{F\subset \mathbb{N}: \sqrt{\min F}\geqslant |F|\}$. 
Clearly, the canonical basis $\mathcal{B}$ is unconditional w.const $\mathbf{K}_u = 1$ and conservative w.const $1$. It is easy to check that $\mathcal{B}$ is PSLC w.const $1$; by Theorem \ref{1pg}, $\mathcal{B}$ is strong partially greedy w.const $1$. 

We now renorm $X$ as follows: choose a fixed weight sequence $\omega = (w(n))_{n=1}^\infty$ satisfying
\begin{itemize}
    \item $0 < w(n) \leqslant 1$ for all $n\in \mathbb{N}$, 
    \item $|\{n: w(n) < 1\}| \leqslant N$ for some $N\in \mathbb{N}$, and 
    \item $(i-j)(w(i)-w(j))< 0$ for some $i<j$. 
\end{itemize}
Define
$$\|(x_1, x_2, x_3, \ldots)\|_\omega\ :=\ \sup_{F\in \mathcal{F}}\sum_{n\in F}w(n)|x_n|.$$
Under $\|\cdot\|_\omega$, the basis $\mathcal{B}$ is not PSLC w.const $1$ because
$$\|e_i\|_\omega \ =\ w(i) \ >\ w(j)\ =\ \|e_j\|_\omega.$$
However, it holds that
$$\gamma_{(N+1) m}(x)\ \leqslant\ \widehat{\sigma}_m(x), \forall x\in X, \forall m\in \mathbb{N}_0.$$
To see this, pick $x = (x_1, x_2, x_3, \ldots)\in X$, $m\in \mathbb{N}$, $k\leqslant m$, and $A\in \mathcal{G}(x,  (N+1)m)$. We need to show that
\begin{equation}\label{e40}\|x-P_A(x)\|_\omega\ \leqslant\ \|x-S_k(x)\|_\omega.\end{equation}
Let $B = \{1, 2, \ldots, k\}\backslash A$, $F = A\backslash \{1, 2, \ldots, k\}$, and $t := \min_{n\in A}|e_n^*(x)|$.

Step 1: We prove the following claim
\begin{equation}\label{e43}\|P_{(A\cup B)^c}(x)+t1_{ B}\|_\omega\ \leqslant\ \left\|P_{(A\cup B)^c}(x) + t1_{ F}\right\|_\omega.\end{equation}
\begin{proof}
Letting $s:=|\{1, 2,\ldots, k\}\cap A|$, we have 
\begin{equation}\label{e41}|B| + N\ =\ k-s + N \ \leqslant\ (N+1)m - s\ =\  |A|-s\ =\ |F|.\end{equation}
Since $|\{n: w(n)<1\}| \leqslant N$, \eqref{e41} implies that $|\{n\in F: w(n) = 1\}|\ \geqslant\ |B|$. Let $D\in \mathcal{F}$ and choose $E\subset \{n\in F: w(n)=1\}$ such that $|E| = |D\cap B|$. Write 
\begin{align*}
    y&\ :=\ P_{(A\cup B)^c}(x)+t1_{ B} \ =:\ (y_1, y_2, \ldots)\\
    z&\ :=\ P_{(A\cup B)^c}(x)+t1_{ F}\ =:\ (z_1, z_2, \ldots). 
\end{align*}
Clearly, $G := (D\cap (A\cup B)^c)\cup E\in \mathcal{F}$ because $|G| \leqslant |D|$ and $\min G\geqslant \min D$. We obtain
\begin{align*}\sum_{n\in D}w(n)|y_n|&\ =\ \sum_{n\in D\cap (A\cup B)^c}w(n)|x_n| + t\sum_{n\in D\cap B}w(n)\\
&\ \leqslant\ \sum_{n\in D\cap (A\cup B)^c}w(n)|x_n| + t\sum_{n\in E}1\\
&\ =\ \sum_{n\in G}w(n) |z_n|\ \leqslant\ \|z\|_\omega.
\end{align*}
Since $D\in \mathcal{F}$ is arbitrary, we conclude that $\|y\|_\omega\leqslant \|z\|_\omega$. 
\end{proof}

Step 2: We now use the claim to prove \eqref{e40}.
We have
\begin{align*}
    \|x-P_A(x)\|_\omega&\ =\ \left\|x-P_A(x)-P_B(x)+\sum_{n\in B}e_n^*(x)e_n\right\|_\omega\\
    &\ \leqslant\  \left\|P_{(A\cup B)^c}(x)+t1_{B}\right\|_\omega\\
    &\ \leqslant\ \left\|P_{(A\cup B)^c}(x) + t1_F\right\|_\omega\\
    &\ \leqslant\ \left\|P_{(A\cup B)^c}(x) + P_F(x)\right\|_\omega\ =\ \left\|x-S_k(x)\right\|_\omega,
\end{align*}
as desired. 
\end{proof}

\begin{rek}\normalfont
The above modification of the Schreier space indicates that though a basis may not be $1$-strong partially greedy, the greedy algorithm is still able to give an error smaller than the error from the $m$th partial sum if we allow greedy sums to be of size greater than $m$ by a constant factor. 
\end{rek}

Note that Theorem \ref{m4} tells us when a basis is $\lambda$-partially greedy w.const $1$.
\begin{cor}\label{c11}
Let $\mathcal B$ be suppression quasi-greedy basis w.const $1$. Then $\mathcal{B}$ is $\lambda$-partially greedy w.const $1$ if and only if it is $\lambda$-PSLC w.const $1$. 
\end{cor}

\begin{exa}\normalfont
We consider the modification of the Schreier space in the proof of Theorem \ref{em3}. The canonical basis $\mathcal{B}$ is suppression unconditional w.const $1$ and thus, is suppression quasi-greedy w.const $1$. With the same notation used in the proof, we now show that $\mathcal{B}$ is $(N+1)$-PSLC w.const $1$. Choose $x\in X$ with $\|x\|_\infty\leqslant 1$, sets $A, B\in\mathbb{N}^{<\infty}$ with $N\max A + |A|\leqslant |B|$ and $ A<\supp(x)\sqcup B$, and signs $\varepsilon, \delta$. We need to show 
\begin{equation}\label{e42}\|x+1_{\varepsilon A}\|_{\omega}\ \leqslant\ \|x+1_{\delta B}\|_\omega.\end{equation}
Since $|B|\geqslant |A| + N\max A\geqslant |A| + N$, we know that 
$|\{n\in B: w(n) = 1\}| \geqslant |A|$. Using the exact argument as in the proof of \eqref{e43} gives \eqref{e42}. By Corollary \ref{c11}, $\mathcal{B}$ is $(N+1)$-partially greedy w.const $1$. 
\end{exa}

\section{Characterizations of strong (reverse) partially greedy bases}
This section aims at providing new characterizations of strong (reverse) partially greedy bases by enlarging greedy sums in a characterization due to Dilworth and Khurana \cite{DK}.

\subsection{Characterizations of strong partially greedy bases}

\begin{defi}\label{d62}\normalfont
A basis $\mathcal B$ is said to be $\lambda$-left-skewed for largest coefficients (LSLC) if there exists a constant $C\geqslant 1$ such that
\begin{equation}\label{e81}\|x+1_{\varepsilon A}\|\ \leqslant\ C\|x+1_{\delta B}\|,\end{equation}
for all $x\in X$ with $\|x\|_\infty\leqslant 1$, for all $A, B\in\mathbb{N}^{<\infty}$ with $A < B$, $\lambda |A|\leqslant |B|$, and $A\sqcup B\sqcup x$, and for all signs $\varepsilon, \delta$. The least constant $C$ verifying \eqref{e81} is denoted by $\mathbf C_{\lambda, lib}$.
\end{defi}

For each $x\in X$, a greedy ordering of $x$ is an injective map $\rho = \rho_x: \mathbb{N}\rightarrow \mathbb{N}$ such that $\supp(x)\subset\rho(x)$ and if $j < k$, then $|e^*_{\rho(j)}(x)|\geqslant |e^*_{\rho(k)}(x)|$. Clearly, $\rho_x$ may not be unique. Let $\mathcal{U}_X = (\rho_x)_{x\in X}$ denote a collection of greedy orderings of $X$. Given a greedy collection $\mathcal{U}_{X}$, we define
$${\sigma}^{\mathcal{U}_X, L}_m(x)\ :=\ \inf\left\{\|x-P_{A}(x)\|\,:\, |A|\leqslant m, \max A < \rho_x(1)\right\}.$$

\begin{defi}\normalfont
A basis $\mathcal{B}$ is said to be $\lambda$-DK partially greedy if there exists $C\geqslant 1$ such that 
$$\left\|x-\sum_{n=1}^{\lceil \lambda m\rceil} e^*_{\rho_x(n)}(x)e_{\rho_x(n)}\right\|\ \leqslant\ C\sigma^{\mathcal{U}_X, L}_m(x), \forall x\in X, \forall m\in \mathbb{N}_0, \forall \mathcal{U}_{X} = (\rho_x)_{x\in X}.$$
The least $C$ is denoted by $\mathbf{C}_{\lambda, dkp}$. Note that when $\lambda = 1$, we retrieve the definition of  partially greedy bases given by Dilworth and Khurana \cite{DK}.
\end{defi}

We gather characterizations of strong partially greedy bases with our introduction of $\lambda$-conservative, $\lambda$-LSLC, and $\lambda$-DK partially greedy. Previous characterizations can be found in \cite[Theorem 2.3]{DK} and \cite[Theorem 3.2]{K}.

\begin{thm}\label{m5}
Let $\mathcal B$ be a basis. The following are equivalent
\begin{enumerate}
    \item[i)] $\mathcal{B}$ is strong partially greedy.
    \item[ii)] $\mathcal{B}$ is quasi-greedy and conservative.
    \item[iii)] $\mathcal{B}$ is quasi-greedy and PSLC. 
    \item[iv)] $\mathcal{B}$ is quasi-greedy and $\lambda$-conservative for all (some) $\lambda\geqslant 1$. 
    \item[v)] $\mathcal{B}$ is quasi-greedy and is $\lambda$-LSLC for all (some) $\lambda\geqslant 1$.
    \item[vi)] $\mathcal{B}$ is $\lambda$-DK partially greedy for all (some) $\lambda\geqslant 1$. 
\end{enumerate}
\end{thm}

\begin{lem}\label{l50}
A basis $\mathcal{B}$ is conservative if and only if $\mathcal{B}$ is ($\lambda$, conservative) for all (some) $\lambda \geqslant 1$.
\end{lem}

The proof of Lemma \ref{l50} is similar to the proof of Lemma \ref{l2}. However, in Lemma \ref{l50}, we do not need the assumption that $\mathcal{B}$ is quasi-greedy. The reason lies in how conservative bases were first defined by Dilworth, Kalton, Kutzarova, and Temlyakov (see \cite[(3.7)]{DKKT}), where the authors allow the cardinality of  the two sets $A$ and $B$ to be different. In contrast, when democracy was first defined by Konyagin and Temlyakov, the authors require the cardinality of the two sets to be the same. Of course, for quasi-greedy bases, whether $|A|\leqslant |B|$ or $|A| = |B|$ makes no difference.

\begin{proof}[Proof of Theorem \ref{m5}]
By \cite[Theorem 4.2]{B0}, we have the equivalence between any pair of i), ii), and iii). Furthermore, Lemma \ref{l50} gives the equivalence between ii) and iv).

To see that v) implies ii), fix $\lambda_0\geqslant 1$. Observe that left $\lambda_0$-SLC gives $\lambda_0$-conservative, which implies conservative by Lemma \ref{l50}. 

We also have that i) implies v). Indeed, by \cite[Theorem 3.2]{K}, we know that a strong partially greedy basis is quasi-greedy and is LSLC and thus, is $\lambda$-LSLC for all $\lambda \geqslant 1$.

We prove that ii) implies  vi). Assume ii) with $\mathcal{B}$ being suppression quasi-greedy w.const $\mathbf C_\ell$. Let $\lambda\geqslant 1$. By \cite[Theorem 2.3]{DK}, we know that $\mathcal{B}$ is $1$-DK partially greedy; that is, there exists $C\geqslant 1$ such that $\forall x\in X, \forall m\in \mathbb{N}_0, \forall \mathcal{U}_{X} = (\rho_x)_{x\in X}$,
    $$\left\|x-\sum_{n=1}^{m} e^*_{\rho_x(n)}(x)e_{\rho_x(n)}\right\|\ \leqslant\ C\sigma^{\mathcal{U}_X, L}_m(x).$$
    For all $x\in X$, $m\in\mathbb{N}_0$, and $\mathcal{U}_{X} = (\rho_x)_{x\in X}$, we, therefore, have
   \begin{align*}\left\|x-\sum_{n=1}^{\lceil \lambda m\rceil} e^*_{\rho_x(n)}(x)e_{\rho_x(n)}\right\|&\ \leqslant\ \mathbf C_\ell\left\|x-\sum_{n=1}^{m} e^*_{\rho_x(n)}(x)e_{\rho_x(n)}\right\|\\
   &\ \leqslant\ \mathbf C_\ell C\sigma^{\mathcal{U}_X, L}_m(x).\end{align*}
    Hence, $\mathcal{B}$ is $\lambda$-DK partially greedy.
 
 It remains to prove that vi) implies ii). Assume that for some $\lambda\geqslant 1$, $\mathcal{B}$ is $\lambda$-DK partially greedy w.const $\mathbf C_{\lambda, dkp}$. Let $x\in X$ and $A\in \mathcal{G}(x, \lceil \lambda m\rceil)$. Choose $\mathcal{U}_X$ such that for $\rho_x\in \mathcal{U}_X$, $A = \{\rho_x(1),\ldots,\rho_x(\lceil \lambda m \rceil)\}$. Then
    $$\|x-P_A(x)\|\ \leqslant\ \mathbf C_{\lambda, dkp}\sigma^{\mathcal{U}_x, L}_{m}(x)\ \leqslant\ \mathbf C_{\lambda, dkp}\|x\|.$$
    By Lemma \ref{l2}, we know that $\mathcal{B}$ is quasi-greedy.
    
    Next, we show that $\mathcal{B}$ is  conservative. Let $A, B\in\mathbb{N}^{<\infty}$ with $\lambda |A|\leqslant |B|$ and $A < B$. Choose $E\subset B$ such that $|E| = \lceil\lambda m\rceil$, where $m:= |A|$. Form $z:= 1_A + 1_{B\backslash E} + 1_{E}$ and choose $\mathcal{U}_{X}$ such that for $\rho_z\in \mathcal{U}_{X}$, $E = \{\rho_z(1), \ldots, \rho_z(\lceil\lambda m\rceil)\}$. By above, we can assume that $\mathcal{B}$ is suppression quasi-greedy w.const $\mathbf C_\ell$. We have
    \begin{align*}\|1_A\|&\ = \ \|z-1_{B\backslash E} - P_E(z)\|\\
    &\ \leqslant\ \mathbf C_\ell \|z-P_E(z)\|\ \leqslant\ \mathbf C_\ell \mathbf C_{\lambda, dkp}\sigma^{\mathcal{U}_X, L}_m(z)
    \ \leqslant\ \mathbf C_\ell \mathbf C_{\lambda, dkp}\|1_B\|.\end{align*}
   Hence, $\mathcal{B}$ is $\lambda$-conservative, which implies that $\mathcal{B}$ is conservative according to Lemma \ref{l50}.

\end{proof}

\subsection{Characterizations of reverse partially greedy bases}

Given a greedy collection $\mathcal{U}_{X}$, we define
$${\sigma}^{\mathcal{U}_X, R}_m(x)\ :=\ \inf\left\{\|x-P_{A}(x)\|\,:\, |A|\leqslant m, \min A > \rho_x(m)\right\}.$$

\begin{defi}\normalfont
A basis $\mathcal{B}$ is said to be $\lambda$-reverse partially greedy if there exists $C\geqslant 1$ such that
$$\left\|x-\sum_{n=1}^{\lceil \lambda m\rceil} e^*_{\rho_x(n)}(x)e_{\rho_x(n)}\right\|\ \leqslant\ C\sigma^{\mathcal{U}_X, R}_m(x), \forall x\in X, \forall m\in \mathbb{N}_0, \forall \mathcal{U}_{X} = (\rho_x)_{x\in X}.$$
The least $C$ is denoted by $\mathbf{C}_{\lambda, rp}$. Note that when $\lambda = 1$, we retrieve the definition of reverse partially greedy bases given by Dilworth and Khurana \cite{DK}.
\end{defi}

\begin{defi}\label{d60}\normalfont
A basis $\mathcal{B}$ is $\lambda$-reverse conservative if for all $A, B\in\mathbb{N}^{<\infty}$ with $\lambda |A| \leqslant |B|$ and $B < A$, we have
$$\|1_A\|\ \leqslant\ C\|1_B\|.$$
The least $C$ is denoted by $\mathbf C_{\lambda, rc}$. When $\lambda = 1$, we say that our basis is reverse conservative.
\end{defi}

\begin{defi}\label{d61}\normalfont
A basis $\mathcal B$ is said to be $\lambda$-right-skewed for largest coefficients (RSLC)  if there exists a constant $C\geqslant 1$ such that
\begin{equation}\label{e80}\|x+1_{\varepsilon A}\|\ \leqslant\ C\|x+1_{\delta B}\|,\end{equation}
for all $x\in X$ with $\|x\|_\infty\leqslant 1$, for all $A, B\in\mathbb{N}^{<\infty}$ with $B < A$, $\lambda |A|\leqslant |B|$ and $A\sqcup B\sqcup x$, and for all signs $\varepsilon, \delta$. The least constant $C$ verifying \eqref{e80} is denoted by $\mathbf C_{\lambda, rib}$. When $\lambda = 1$, we say that $\mathcal{B}$ is RSLC. 
\end{defi}

\begin{thm}\label{m8}
Let $\mathcal{B}$ is a basis. The following are equivalent
\begin{enumerate}
    \item[i)] $\mathcal{B}$ is reverse partially greedy.
    \item[ii)] $\mathcal{B}$ is $\lambda$-reverse partially greedy for all (some) $\lambda \geqslant 1$.
    \item[iii)] $\mathcal{B}$ is quasi-greedy and is reverse conservative.
    \item[iv)] $\mathcal{B}$ is quasi-greedy and is RSLC.
    \item[v)] $\mathcal{B}$ is quasi-greedy and is $\lambda$-reverse conservative for all (some) $\lambda\geqslant 1$.
    \item[vi)] $\mathcal{B}$ is quasi-greedy and has $\lambda$-RSLC for all (some) $\lambda\geqslant 1$.
\end{enumerate}
\end{thm}

Using the same argument as in the proof of Lemma \ref{l2}, we have
\begin{lem}\label{l70}
A basis $\mathcal{B}$ is reverse conservative if and only if $\mathcal{B}$ is $\lambda$-reverse conservative for all (some) $\lambda\geqslant 1$. 
\end{lem}

\begin{proof}[Proof of Theorem \ref{m8}]
By \cite[Theorem 2.7]{DK} and \cite[Theorem 3.12]{K}, we have the equivalence between any pair of (i), (iii), and (iv). The equivalence between (iii) and (v) follows from Lemma \ref{l70}.

We prove that (i) implies (ii). Assume that $\mathcal{B}$ is reverse partially greedy w.const $\mathbf C_{1, rp}$. Fix $\lambda > 1$. We need to show that $\mathcal{B}$ is $\lambda$-reverse partially greedy. Since i) is equivalent to iii), we can assume that $\mathcal{B}$ is suppression quasi-greedy w.const $\mathbf{C}_\ell$. Pick an arbitrary collection $\mathcal{U}_{X} = (\rho_x)_{x\in X}$. Let $x\in X$ and $m\in\mathbb{N}_0$. We have
    \begin{align*}\left\|x-\sum_{n=1}^{\lceil\lambda m\rceil}e^*_{\rho_x(n)}(x)e_{\rho_x(n)}\right\|&\ \leqslant\ \mathbf C_\ell\left\|x-\sum_{n=1}^{m}e^*_{\rho_x(n)}(x)e_{\rho_x(n)}\right\|\\
    &\ \leqslant\ \mathbf C_\ell\mathbf C_{1, rp}\sigma_m^{\mathcal{U}_{X}, R}(x).\end{align*}
    Therefore, $\mathcal{B}$ is $\lambda$-reverse partially greedy.

To see that ii) implies iii), we assume that $\mathcal{B}$ is $\lambda$-reverse partially greedy w.const $\mathbf C_{\lambda, rp}$ for some $\lambda\geqslant 1$. Let $x\in X$ and $A\in \mathcal{G}(x, \lceil \lambda m\rceil)$. Choose $\mathcal{U}_X$ such that for $\rho_x\in \mathcal{U}_X$, $A = \{\rho_x(1),\ldots,\rho_x(\lceil \lambda m \rceil)\}$. Then
    $$\|x-P_A(x)\|\ \leqslant\ \mathbf C_{\lambda, rp}\sigma^{\mathcal{U}_x, R}_{m}(x)\ \leqslant\ \mathbf C_{\lambda, rp}\|x\|.$$
    By Lemma \ref{l2}, we know that $\mathcal{B}$ is quasi-greedy. Next, we show that $\mathcal{B}$ is reverse conservative. Let $A, B\in\mathbb{N}^{<\infty}$ with $\lambda |A|\leqslant |B|$ and $B < A$. Choose $E\subset B$ such that $|E| = \lceil\lambda m\rceil$, where $m:= |A|$. Form $z:= 1_A + 1_{B\backslash E} + 1_{E}$ and choose $\mathcal{U}_{X}$ such that for $\rho_z\in \mathcal{U}_{X}$, $E = \{\rho_z(1), \ldots, \rho_z(\lceil\lambda m\rceil)\}$. By above, we can assume that $\mathcal{B}$ suppression quasi-greedy w.const is $\mathbf C_\ell$. We have
    \begin{align*}\|1_A\|\ = \ \|z-1_{B\backslash E} - P_E(z)\|&\ \leqslant\ \mathbf C_\ell \|z-P_E(z)\|\\
    &\ \leqslant\ \mathbf C_\ell \mathbf C_{\lambda, rp}\sigma^{\mathcal{U}_X, R}_m(z)
    \ \leqslant\ \mathbf C_\ell \mathbf C_{\lambda, rp}\|1_B\|.\end{align*}
   Hence, $\mathcal{B}$ is $\lambda$-reverse conservative, which implies that $\mathcal{B}$ is reverse conservative according to Lemma \ref{l70}.
   
  Furthermore, that iv) implies vi) is immediate from the definition of $\lambda$-RSLC. 
 
 Finally,  if $\mathcal{B}$ is $\lambda_0$-RSLC for some $\lambda_0\geqslant 1$, then $\mathcal{B}$ is $\lambda_0$-reverse conservative, which implies reverse conservative by Lemma \ref{l70}. Hence, vi) implies iii).
\end{proof}

\begin{rek}\normalfont
As we have seen, for $\lambda > 1$, ($\lambda$, partially greedy) is strictly weaker than strong partially greedy. In a similar manner, is there a notion of ($\lambda$, reverse partially greedy) that is strictly weaker than reverse partially greedy? Recently, the author gives a positive answer in \cite{Chu}.
\end{rek}

\ \\
\end{document}